\documentclass{article} 
\usepackage{times,authblk}
\usepackage[margin=1in]{geometry}
\usepackage{mkolar_definitions}
\usepackage{natbib,rotating,subfigure}

\title{Changepoint Detection over Graphs with the\\ Spectral Scan Statistic}

\author[1,2]{
James Sharpnack
\thanks{jsharpna@cs.cmu.edu}}
\author[2]{
Alessandro Rinaldo
\thanks{arinaldo@cmu.edu}}
\author[1]{
Aarti Singh
\thanks{aarti@cs.cmu.edu}
}
\affil[1]{Machine Learning Department\\
Carnegie Mellon University}
\affil[2]{Statistics Department\\
Carnegie Mellon University}

\begin{document}

\maketitle

\begin{abstract}
We consider the change-point detection problem of deciding, based on noisy measurements, whether an unknown signal over a given graph is constant or is instead piecewise constant over two connected induced subgraphs of relatively low cut size. We analyze the corresponding generalized likelihood ratio (GLR) statistics and relate it to the problem of finding a sparsest cut in a graph. We develop a tractable relaxation of the  GLR statistic based on the combinatorial Laplacian of the graph, which we call the spectral scan statistic,  and analyze its properties. We show how its performance as a testing procedure depends directly on the spectrum of the graph, and use this  result to explicitly derive its asymptotic properties on few significant graph topologies. Finally, we demonstrate both theoretically and by simulations that the spectral scan statistic can outperform naive testing procedures based on edge thresholding and $\chi^2$ testing. 
\end{abstract}


\section{Introduction}


In this article we are concerned with the basic but fundamental task of deciding whether a given graph, over which a noisy signal is observed, contains a cluster of anomalous or activated nodes comprising an induced connected subgraph. Such a problem is highly relevant in a variety of scientific areas, such as community detection in social networks, surveillance,  disease outbreak detection, biomedical imaging, sensor network detection, gene network analysis, environmental monitoring and malware detection. 
Recent theoretical contributions in the statistical literature (see, e.g., \cite{castro:05, arias2008searching, arias2011detection,addario2010combinatorial}) have detailed the inherent difficulty of such a testing problem in relatively simplified settings and under specific conditions on the graph topology.
From a practical standpoint, the natural algorithm for  detection of anomalous clusters of activity in graphs  is the the generalized likelihood ratio test (GLRT)  or scan statistic,  a computationally intensive procedure that entails scanning all well connected clusters and testing individually for anomalous activation. 
Unfortunately, its performance over general graphs is not well understood, and little attention has been paid to determining alternative, computationally tractable, procedures.

In this article we assume that the class of clusters of activation consists of sub-graphs of small cut size. We believe this is a natural and realistic assumption which, as we demonstrate below,  allows us to explicitly incorporate into the detection problem the properties of the graph topology through its spectrum.
In particular, we show that the GLRT is an integer program with a term in the objective that corresponds to the sparsest cut in a graph, a known NP-hard problem.
With this in mind, we propose a relaxation of the GLRT, called the spectral scan statistic, which is based on the combinatorial Laplacian of the graph and, importantly, is a tractable program.
As our main result, we derive theoretical guarantees for the performance of the spectral scan statistic, which hold for any graph and are based on the spectral measure of the combinatorial Laplacian.
For comparison purposes, we derive theoretical guarantees for two simple estimators, the edge thresholding and the $\chi^2$ test.
We conclude our study by applying the main result to balanced binary trees, the lattice, and Kronecker graphs, giving us precise asymptotic results.
We find that, modulo logarithm terms, the spectral scan statistic has nearly optimal power for balanced binary trees.
Simulations for these models verify that the spectral scan statistic dominates the simple estimators.

{\bf Contributions}. Our contributions are as follows. (1) We define a new class of activation patterns based on the notion of small cut size that reflects in a natural way the topological properties of the graph. (2) We analyze the corresponding GLR statistics and show that it is indeed related to the problem of finding sparest cuts. We then develop a computationally tractable relaxation of the GLR statistic, called the spectral scan statistic and analyze its properties. In our main theoretical result, we show show that the performance of the spectral scan statistic depends explicitly on the spectral properties of the graph. (3) Using such results we are able to characterize in a very explicit form the performance of the  spectral scan statistic on a few notable graph topologies and demonstrate its superiority over naive detectors, such as the edge thresholding and the $\chi^2$ test. (4) Finally, we have formulated the detection problem under more general and realistic scenarios, which involve composite null and alternative hypotheses as opposed to simple hypotheses as is customary in the theoretical statistical literature on this subject.

{\bf Related Work.}
Normal means testing in high-dimensions is a well established and fundamental problem in statistics (see, e.g., \cite{ingster2003nonparametric}). A significant portion of the recent work in this area (\cite{castro:05, arias2008searching, arias2011detection,addario2010combinatorial}) has focused on incorporating structural assumptions on the signal, as a way to mitigate the effect of high-dimensionality and also because many real-life problems can be represented as instances of the normal means problem with graph-structured signals (see, for an example, \cite{jacob2010gains}).
These contributions have considered the generalized likelihood ratio test of means when the alternative hypothesis takes on the form of a combinatorial space. However, the  performance of  such test has been analyzed only for certain types of graphs, and it is unclear to what extent those analyses extend to general graph topologies.
Moreover, while much is known about the theoretical performance of the GLRT, no mention is made about its computational feasibility.
Another line of research relevant to our problem is the  optimal fail detection with nuisance parameters and matched subspace detection in the signal processing literature: see, e.g. \cite{scharf94,baygun.hero:95,fouladirad2005optimal,fouladirad2008optimal}. Though our problem can be cast as a special case of the more general problem of optimal testing of a linear subspace under nuisance parameters considered in that line of work, the focus on a graph-structured signal, as well as the type of analysis based on the interplay between the scan statistics and the spectral properties of the graph contained in our work, are novel. 


\subsection{Problem Setup}
We now formalize the problem of detecting a change of signal over the vertices of a graph from noisy observations in the high-dimensional setting. 
For a given connected, undirected, possibly weighted graph $G = (V,E)$ on $|V|=n$ nodes, we observe {\it one} realization of the random vector
\begin{equation}\label{eq:model}
\yb = \betab + \epsilonb,
\end{equation}
where $\betab \in \RR^V$ and $\epsilonb \sim N(0,\sigma^2 \Ib_n)$, with  $\sigma^2$ known. 
We will assume that there are two groups of constant activation for the signal $\betab$, namely that there exists a subset  $C \subset V$ such that $\betab$ is constant within both $C$ and it complement $\bar C = V \backslash C$. We formalize this assumption by writing
\begin{equation}\label{eq:beta.signal}
\betab = \mu {\bf 1} + \delta {\bf 1}_C,
\end{equation}
where $\mu, \delta \in \mathbb{R}$ are unknown parameters, ${\bf 1} \in \mathbb{R}^V$ is a $n$-dimensional vector of ones and ${\bf 1}_C$ is the indicator function of the subset $C$. The parameter $\mu$ can be thought of as the magnitude of the background signal and is a nuisance parameter, while $\delta$ quantifies the  the gap in signal between the two clusters. Setting $\bar \betab = {\bf 1}^\top \betab / n$, we will use $\| \betab - \bar \betab \|$ to measure the energy of the signal (note that this quantity is independent of $\mu$), and we will define the signal-to-noise ratio (SNR) to be
\[
\frac{\| \betab - \bar \betab \|}{\sigma} =  \sqrt{\frac{|C||\bar C|}{n}} \frac{\delta}{\sigma}.
\]
We will not assume any knowledge of the true clustering $(C,\bar C)$, other than that it belongs to a given class $\mathcal{C}$ of bi-partitions $(C,\bar C)$ of $V$ such that $C$ and $\bar{C}$ are both large and can be easily disconnected, in that they have low cut size . 
Formally, we define, for some $\rho > 0$,
\begin{equation}
\label{eqn:Cclass}
\Ccal = \Ccal(\rho) =  \left\{C \subset V, C \neq \emptyset \colon \frac{ |\partial C| }{ |C||\bar C|} \le \frac {\rho} {|V|} \right\},
\end{equation}
where $\partial C = \{(i,j) \in E: i \in C, j \in \bar C \}$ is  the boundary of $C$. Note that $\mathcal{C}$ is a symmetric class in the sense that $C \in \mathcal{C}$ if and only if $\bar C \in \mathcal{C}$. 
We are interested in the problem of testing whether the gap parameter $\delta$ in equation \eqref{eq:beta.signal} is zero (i.e. the signal $\betab$ is constant) or it is non-zero for some $C \in \mathcal{C}$, regardless of the value of $\mu$. Thus, we can naturally cast our structured change-point detection problem as the following composite hypothesis testing problem:
\begin{equation}\label{eq:H0:H1}
H_0 \colon \betab \in \Theta_0 \quad \textrm{vs} \quad H_1 \colon \betab \in \Theta_1, 
\end{equation}
where $\Theta_0 = \{ \mu {\bf 1}, \mu \in \mathbb{R}\}$ and $\Theta_1 = \{  {\bf 1} \mu + {\bf 1}_C \delta, \mu \in \mathbb{R}, \delta \in \mathbb{R} \setminus \{ 0\}, C \in \mathcal{C}\}$. Notice that the alternative can be written as the join over $\mathcal{C}$ of disjoint composite alternatives of the form $
H_1^C \colon \betab \in \Theta^C_1 : = \{  {\bf 1} \mu + {\bf 1}_C \delta, \mu \in \mathbb{R}, \delta \in \mathbb{R} \setminus \{ 0\} \} 
$, $C \in \mathcal{C}$.

To make our analysis meaningful, we measure the difficulty of the detection problem in terms of the energy parameter by assuming that, for some $\eta > 0$, $\| \betab - \bar \betab\| > \eta, \quad \forall \betab \in \Theta_1$.
Thus, we can think of $\eta$ as the minimal degree of separation between the null and alternative hypotheses. 
Below we will analyze asymptotic conditions under which the hypothesis testing problem described above is feasible, in a sense made precise in the next definition, 
when the size of the graph $n$ increases unboundedly. To this end, we will further assume  that the relevant parameters of the model, $\eta$, $\sigma$, $\delta$ and $\rho$ change with $n$ as well, even though we will not make such dependence explicit in our notation for ease of readability.  
Our results establish conditions for asymptotic disinguishability as a function of the SNR $\eta/\sigma$ and $\rho$ and the spectrum of the graph $G$.

\begin{definition}
Let $P_{\theta}$ denote the distribution of ${\bf y}$ induced by the model \eqref{eq:model}, where $\theta \in \Theta_0 \cup \Theta_1$. For a given statistic $S({\bf y})$ and threshold $\tau \in \mathbb{R}$, let $T = T({\bf y})$ be $1$ if $S({\bf y}) > \tau$ and $0$ otherwise. We say that the hypotheses $H_0$ and $H_1$ are {\bf asymptotically distinguished by the test} $T$ if
\begin{equation}
    \label{eqn:asymp_dist}
  \sup_{\theta \in H_0} \PP_\theta \{ T=1 \} \rightarrow 0 \quad \textrm{ and } \quad \sup_{\theta \in H_1} \PP_\theta \{ T=0 \} \rightarrow 0, 
  \end{equation}
  where the limit is taken as $n \rightarrow \infty$.
We say that $H_0$ and $H_1$ are {\bf asymptotically indistinguishable} if there does not exist any test for which the above limits hold.
\end{definition}

{\bf Notation}. We will need some mathematical terminology from algebraic graph theory (\cite{godsil2001algebraic}).
A central object to our analysis is the {\em combinatorial Laplacian} matrix $\Lb = \Db - \Wb$, where 
 $\Wb = (I\{(v,w) \in E\})_{v,w \in V}$ is the adjacency matrix of the graph $G$ and  $\Db = \diag\{ d_v\}_{v \in V}$ is the diagonal matrix of node degrees, $d_v = \sum_{w \in V} W_{v,w}$, $v \in V$.  
If the graph is weighted then $W_{v,w}$ reflects this.
We will denote the eigenvalues of $\Lb$ with $\{\lambda_i \}_{i = 1}^n$, which we will always take in increasing order.
Since $G$ is connected, the smaller eigenvalue $\lambda_1 = 0$, with corresponding eigenvector, $\one$.
$\lambda_2$ is known as the {\em algebraic connectivity} and is lower bounded by $4[n \textrm{diam}(G)]^{-1}$ where $\textrm{diam}(G)$ is the diameter of the graph.
Throughout this study we use Bachmann-Landau notation for asymptotic statements: if $a_n/b_n \rightarrow 0$ then $a_n = o(b_n)$ and $b_n = \omega(a_n)$.

\section{Methods}

The hypothesis testing problem at hand presents two challenges: (1) the model contains an unbounded nuisance parameter $\mu \in \mathbb{R}$ and (2) the alternative hypothesis is comprised of a finite disjoint union of composite hypotheses indexed by $\mathcal{C}$. These features set our problem apart from virtually all existing work of structured normal means problems (see, e.g. \cite{castro:05, arias2008searching, arias2011detection,addario2010combinatorial}), which does not consider	 nuisance parameters and relies on a simplified framework consisting of a simple null hypothesis and a composite hypothesis consisting of disjoint unions of simple alternatives. Having nuisance parameters and composite hypothesis require a more sophisticated analysis.

We will eliminate the interference caused by the nuisance parameter by considering test procedures that are independent of $\mu$. The formal justification for this choice is based on the theory of optimal invariant hypothesis testing (see, e.g., \cite{lehmann2005testing}) and of uniformly best constant  power tests (see \cite{wald:43}). Due to space limitations we will not provide the details and refer the reader to \cite{fouladirad2008optimal,fouladirad2005optimal,fillatre2007non,fillatre2012,scharf94,baygun.hero:95} and references therein for in depth-treatments of these issues related to the model a hand.

For the simpler problem of testing $H_0$ versus $H_1^C$ for some $C \subset V$, the optimal test is based on the likelihood ratio (LR) statistic (see the proof of Lemma \ref{lem:GLRTform} below for a derivation)
\begin{equation}\label{eq:LR}
2 \log \Lambda_C(\yb) = \log\left( \frac{\sup_{\theta \in \Theta_1} f_\theta(\yb)}{\sup_{\theta \in \Theta_0} f_\theta(\yb)} \right) =  \frac{1}{\sigma^2} \frac{|V|}{|C| |\bar{C}|} \left( \sum_{v \in C} \tilde{\yb}_v \right)^2,
\end{equation}
where $\tilde \yb = \yb - \bar \yb = (\tilde \yb_v, v \in V)$ and $f_\theta$ is the Lebesgue density of $P_\theta$. This test rejects $H_0$ for large values of $\Lambda_C(\yb)$. Optimality follows from the fact that the statistical model we consider has the monotone likelihood ratio property.

When testing against composite alternatives, like in our case, it is customary to consider instead the generalized likelihood ratio (GLR) statistic, which in our case reduces to
\[
\hat g =  \max_{C \in \mathcal{C}(\rho)} 2 \sigma^2 \log \Lambda_C(\yb).
\]
Through manipulations of the likelihoods, we find that the GLR statistic has a very convenient form which is tied to the spectral properties of the graph $G$ via its Laplacian.
\begin{lemma}
\label{lem:GLRTform}
  Let $\tilde{\yb} = \yb - \one (\frac{1}{n} \sum_{v \in V} \yb_v)$ and $\Kb = \Ib - \frac{1}{n} \one \one^\top$. Then
\begin{equation}
\label{eqn:GLRT}
\hat g = \max_{\xb \in \{0, 1\}^n} \frac{\xb^\top \tilde{\yb}\tilde{\yb}^\top \xb}{\xb^\top \Kb \xb} \textrm{ s.t. } \frac{ \xb^\top \Lb \xb}{\xb^\top \Kb \xb} \le \rho,
\end{equation}
where $\Lb$ is the combinatorial Laplacian of the graph $G$.
\end{lemma}

The proof is provided in the appendix.
The savvy reader will notice the connection between \eqref{eqn:GLRT} and the graph sparsest cut program.
By Lagrangian duality, we see that the program \eqref{eqn:GLRT} is equivalent to (for some Lagrangian parameter $\nu$) 
\[
\min_{C \subseteq V} \frac{|\partial C |}{|C||\bar C|} - \nu \frac{(\sum_{i \in C} \tilde y_i )^2}{|C||\bar C|}
\]
the first term of which is precisely the {\em sparsest cut} objective, and the second term drives the solution $C$ to have positive within cluster empirical correlations.
The sparsest cut program is known to be NP-hard, with poly-time algorithms known for trees and planar graphs(\cite{matula1990sparsest}).
Because of this fact, approximate algorithms have been proposed over the past two decades, most notably the uniform multicommodity flow approach of (\cite{leighton1988approximate,shmoys1997cut}) and the semi-definite relaxation of the cut metric (\cite{arora2009expander}).
\cite{hagen1992new} observed that the minimum cut sparsity is bounded by the algebraic connectivity ($\lambda_2$), suggesting the Fiedler vector (i.e. the second eignenvector of $\Lb$) to be an appropriate relaxation of the  characteristic vector of the cut.
Moreover, the well known Cheeger inequality shows that the minimum cut sparsity (in a regular graph) is bounded by the algebraic connectivity (see \cite{chung2004discrete}).
We will follow the tradition of bounding sparsity with the algebraic connectivity, and provide a surrogate estimator to the scan statistic based on this simple spectral relaxation.

\begin{proposition}
Define the Spectral Scan Statistic (SSS) as 
\[
\hat{s} = \sup_{\xb \in \RR^n} (\xb^\top \tilde\yb)^2  \textrm{ s.t. }\xb^\top \Lb \xb \le \rho, \| \xb \| \le 1, \xb^\top \one = 0.
\]
Then the GLR statistic is bounded by the SSS: $\hat g \le \hat s$.
\end{proposition}

\begin{proof}
First let us notice that $\Kb = \Ib - \frac{1}{n} \one \one^\top$ is the projection onto the subspace orthogonal to $\one$.  Because $\Kb$ is thus idempotent, $\tilde \yb \one = 0$, and $\Lb \one = 0$ we can rewrite 
\[
\hat g = \max_{\xb \in \{0, 1\}^n\backslash \{\zero,\one \} }\frac{(\Kb \xb)^\top \tilde{\yb}\tilde{\yb}^\top (\Kb \xb)}{(\Kb \xb)^\top (\Kb \xb)}  \textrm{ s.t. } \frac{(\Kb \xb)^\top \Lb (\Kb \xb)}{(\Kb \xb)^\top (\Kb \xb)} \le \rho
\]
So, we have the following relaxation,
\[
\hat g \le \max_{\xb \ne 0, \xb^\top \one = 0} \frac{\xb^\top \tilde{\yb}\tilde{\yb}^\top \xb}{\xb^\top \xb}  \textrm{ s.t. } \frac{\xb^\top \Lb \xb}{\xb^\top \xb} \le \rho = \hat s
\]
\end{proof}

\begin{remark}
By Lagrangian duality and the Courant-Fischer theorem, the spectral scan statistic can be written as
\[
\hat s = \min_{\nu > 0} \chi( \tilde \yb \tilde \yb^\top - \nu \Delta) + \nu \rho
\]
where $\chi(A)$ is the maximum non-zero eigenvalue of the matrix $A$.
\end{remark}

Notice that because the domain $\Xcal = \{ \xb \in \RR^n : \xb^\top \Lb \xb \le \rho, \| \xb \| \le 1, \xb^\top \one = 0 \}$ is symmetric around the origin, this is precisely the square of the solution to
\begin{equation}
\label{eqn:sGP}
\sqrt{\hat{s}} = \sup_{\xb \in \RR^n} \xb^\top \yb  \textrm{ s.t. }\xb^\top \Lb \xb \le \rho, \| \xb \| \le 1, \xb^\top \one = 0,
\end{equation}
where we have used the fact that $\xb^\top \tilde\yb = ( (\Ib - \frac{1}{n}\one \one^\top) \xb )^\top \yb = \xb^\top \yb$ because $\xb^\top \one = 0$ within $\Xcal$. This previous formulation shows that the SSS is related to the supremum of a Gaussian process over $\mathcal{X}$. This fact will turn out to be extremely convenient, as we show next.
\section{Theoretical Analysis}

We first derive a simple condition for asymptotic indistinguishability based on testing the null versus a single component in the alternative. 
A more refined analysis of the lower bound for the general hypothesis \eqref{eq:H0:H1} is beyond the scope of this article. 

\begin{theorem}
\label{thm:lower_bd}
Suppose that there exists $C \in \Ccal$ such that $\frac{|\bar C|}{|C|} \asymp 1$. 
Then $H_0$ and $H_1$ are asymptotically indistinguishable if $\eta / \sigma = o(1)$.
\end{theorem}

The proof is in the appendix.
We will analyze the performance of the SSS statistic by relying on its representation \eqref{eqn:sGP} as the square of the supremum of a Gaussian process. 
We draw heavily on the theory of the generic chaining, perfected in \cite{talagrand2005generic}, which essentially reduces the problem of computing bounds on the expected supremum of Gaussian processes to geometric properties of its index space.
Recall that, under alternative hypothesis, $\| \betab - \bar \betab\| \geq \eta$ uniformly over $\Theta_1$.

\begin{theorem}
  \label{thm:hypo_bd}
  The following hold with probability at least $1 - \delta$.
  Under the null $H_0$
  \[
  \hat s \le \left( \sqrt{2 \sigma^2 \sum_{i > 1} \min\{1, \rho \lambda_i^{-1}\}} + \sqrt{2 \sigma^2 \log\frac{2}{\delta}} \right)^2,
  \]
  while the alternative $H_1$ 
  \[
  \hat s \ge \left( \eta - \sqrt{2 \sigma^2 \log\frac{2}{\delta}} \right)^2.
  \]
\end{theorem}

\begin{proof}
We use generic chaining to control the process $\{ \xb^\top \yb \}_{\xb \in \Xcal}$ appearing in the SSS. 
First, we notice that the index set $\Xcal$ is the intersection of an ellipsoid and the unit ball, which is the intuition behind the following lemma.
\begin{lemma}
  \label{lem:cheegerBdd}
  Let $\Lb$ have spectrum $\{\lambda_i \}_{i=1}^n$. Then under $H_0$,
  \[
  \EE \sup_{\xb \in \Xcal} \xb^\top \yb \le \sqrt{2 \sigma^2 \sum_{i > 1} \min\{1,\rho \lambda_i^{-1}\}}.
  \]
\end{lemma}

The proof is provided in the appendix.
We then can use the well known phenomena, that the supremum of a Gaussian process concentrates around it's expectation (see the appendix).
Hence, by Lemma \ref{lem:gauss_sup_conc} the first statement in Theorem \ref{thm:hypo_bd} holds.
The second statement follows by applying standard concentration results to the univariate Gaussian  $\frac{\betab - \bar \betab}{\| \betab - \bar \betab \|} \yb$ and noticing that $\frac{\betab - \bar \betab}{\| \betab - \bar \betab \|} \in \Xcal$ and $\EE \frac{(\betab - \bar \betab)^\top}{\| \betab - \bar \betab \|} \yb = \| \betab - \bar \betab \| \ge \eta$ under $H_1$.
\end{proof}

As a corollary we will provide sufficient conditions for asymptotic distinguishability that depend on the spectrum of the Laplacian $\Lb$. As we will show in the next section, these conditions can be applied to a number of graph topologies whose spectral properties are known. 

\begin{corollary}
  The null and alternative, as described in Thm.~\ref{thm:hypo_bd}, are asymptotically distinguished by $\hat s$ and $g_\nu(\yb)$ if
  \begin{equation}
    \label{eqn:dist_bound1}
  \frac{\eta}{\sigma} = \omega \left(\sqrt{\sum_{i > 1} \min\{1, \rho \lambda_i^{-1}\}} \right)
  \end{equation}
  Other stronger sufficient conditions are
  \begin{equation}
    \label{eqn:dist_bound2}
  \frac{\eta}{\sigma} = \omega \left( \sqrt{ k + \frac{(n - k) \rho}{\lambda_{k + 1}} } \right)
  \end{equation}
  if $k$ is large enough that $\lambda_{k+1} > \rho$.
\end{corollary}

\begin{proof}
To see equation~\eqref{eqn:dist_bound1} we note that, due to Theorem \ref{thm:hypo_bd}, if 
\[
\sqrt{2 \sigma^2 \sum_{i > 1} \min\{1, \rho \lambda_i^{-1}\}} + \sqrt{2 \sigma^2 \log\frac{2}{\delta}} = o \left(\eta - \sqrt{2 \sigma^2 \log\frac{2}{\delta}} \right)
\]
then we attain asymptotic distinguishability by choosing any threshold $\tau$ between, and sufficiently far from, the left and right hand side of the previous display.
To show equation~\eqref{eqn:dist_bound2} we note that by choosing $k$ such that $\lambda_{k+1} > \rho$ we see that 
\[
\sum_{1 < i \le k} \min\{1, \rho \lambda_i^{-1}\} \le k \Rightarrow
\sum_{i > k} \min\{1, \rho \lambda_i^{-1}\} \le (n - k) \frac{\rho}{\lambda_{k+1}}.
\vspace{-.2in}
\]
\vspace{-.1in}
\end{proof}

Interestingly, there are no logarithmic terms in \eqref{eqn:dist_bound1} that usually accompany uniform bounds of this type, which is attributed to the generic chaining.
Notice that the left hand side of \eqref{eqn:dist_bound1} is always less than $\sqrt{n-1}$, which we will see characterizes the performance of the naive estimator $\| \tilde \yb \|$.


For comparison, we consider the performance of two naive procedure for detection: the energy detector, which reject $H_0$ if $\| \tilde \yb\|^2$ is too large and the edge thresholding detector, which reject $H_0$ if $\max_{(v,w) \in E} | \yb_v - \yb_w |$ is large.
\begin{theorem}
\label{thm:energy}
 $H_0$ and $H_1$ are asymptotically distinguished by $\| \tilde\yb \|$ if and only if 
\[
\frac{\eta}{\sigma} = \omega(\sqrt{n-1}).
\]  
\end{theorem}
The proof (given in the appendix) is a standard $\chi^2$ analysis.
In \cite{sharpnacksparsistency} the authors examined the problem of exact recovery of cluster boundaries in the graph-structured normal means problem by taking differences between observations corresponding to adjacent nodes.
The following result stems from Theorem 2.1 of \cite{sharpnacksparsistency}, and the fact that $|C||\bar C| / n$ scales like $\min\{|C|,|\bar C|\}$ up to a factor of $2$.
\begin{theorem}
\label{thm:edge_thresh}
$H_0$ and $H_1$ are asymptotically distinguished by $\max_{(v,w) \in E} | \yb_v - \yb_w |$ if
\[
\frac{\eta}{\sigma} = \omega \left(\sqrt{\max_{C \in \Ccal, |C| \le n/2} |C| \log n } \right).
\]
\end{theorem}
If $\Ccal$ contains balanced clusters, i.e. bipartitions $(C,\bar C)$ such that $\frac{|C|}{|\bar C|} \asymp 1$, then this result matches the scaling in Theorem \ref{thm:energy} up to a log factor.


\section{Specific Graph Models}
In this section we demonstrate the power and flexibility of Theorem \ref{thm:hypo_bd} by analyzing in detail the performance of the spectral scan statistic over three important graph topologies: balanced binary trees, the s-dimensional lattice and the Kronecker graphs (see \cite{leskovec2007scalable,leskovec2010kronecker}).

\subsection{Balanced Binary Trees}
We begin the analysis of the spectral scan statistic by applying it to the balanced binary tree (BBT) of depth $\ell$.
The class of signals that we will consider have clusters of constant signal which are subtrees of size at least $c n^\alpha$ for $0<c\le 1/2, 0<\alpha\le 1$.
Hence, the cut size of the signals are $1$ and $\rho = [cn^\alpha(1 - cn^{\alpha-1})]^{-1}$.

\begin{corollary}
\label{cor:BBT}
For the balanced binary tree with $n$ vertices, the spectral scan statistic can asymptotically distinguish $H_0$ from signals with $\rho = n[cn^\alpha(n - cn^\alpha)]^{-1}$ if the SNR is stronger than
\[
\frac{\eta}{\sigma} = \omega ( n^\frac{1-\alpha}{2} \log n ).
\]
\end{corollary}

We simulate the probability of correct discovery of change-points (rejecting $H_0$ when the truth is $H_1$) versus the probability of false alarm (falsely rejecting $H_0$).
These are given for the four estimators in Figure \ref{fig1} and for the SSS as $n = 2^{\ell + 1} - 1$ increases.
In these simulations a subtree at level $2$ (of size $n/4$) was chosen as $C$, the gap-to-noise ratio is fixed at $\delta/\sigma = 0.8$, and $\rho = 4/n$.
We see that even in the low $n$ regime, exploiting the graph structure is essential to improve the power of testing $H_0$ against $H_1$.
As $n$ increases with $\delta/\sigma$ fixed the performance of the SSS dramatically increases.

\begin{figure}[h]
\centering
\mbox{
\subfigure{\includegraphics[width=2.1in]{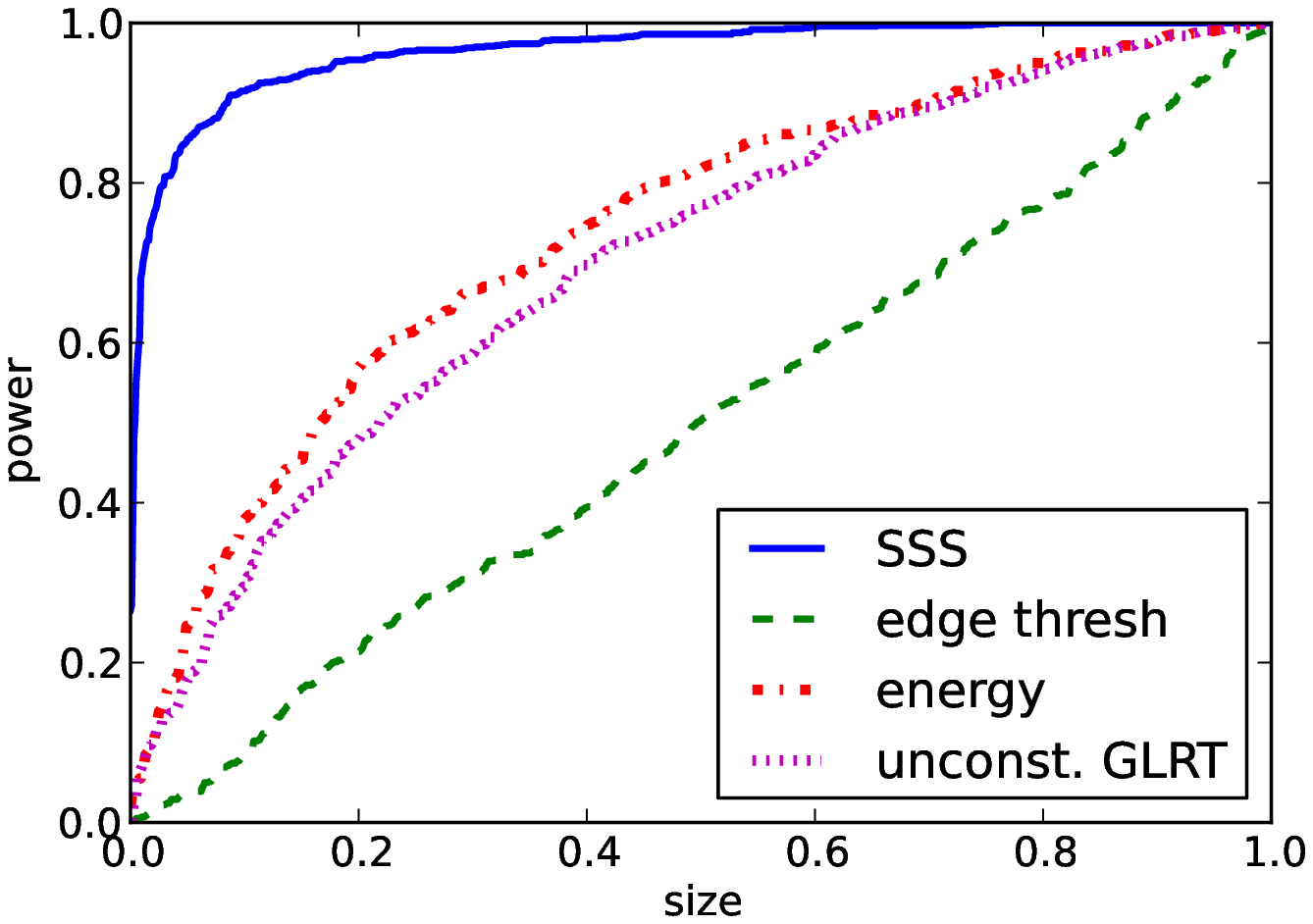}}\hspace{-.1in}
\subfigure{\includegraphics[width=2.1in]{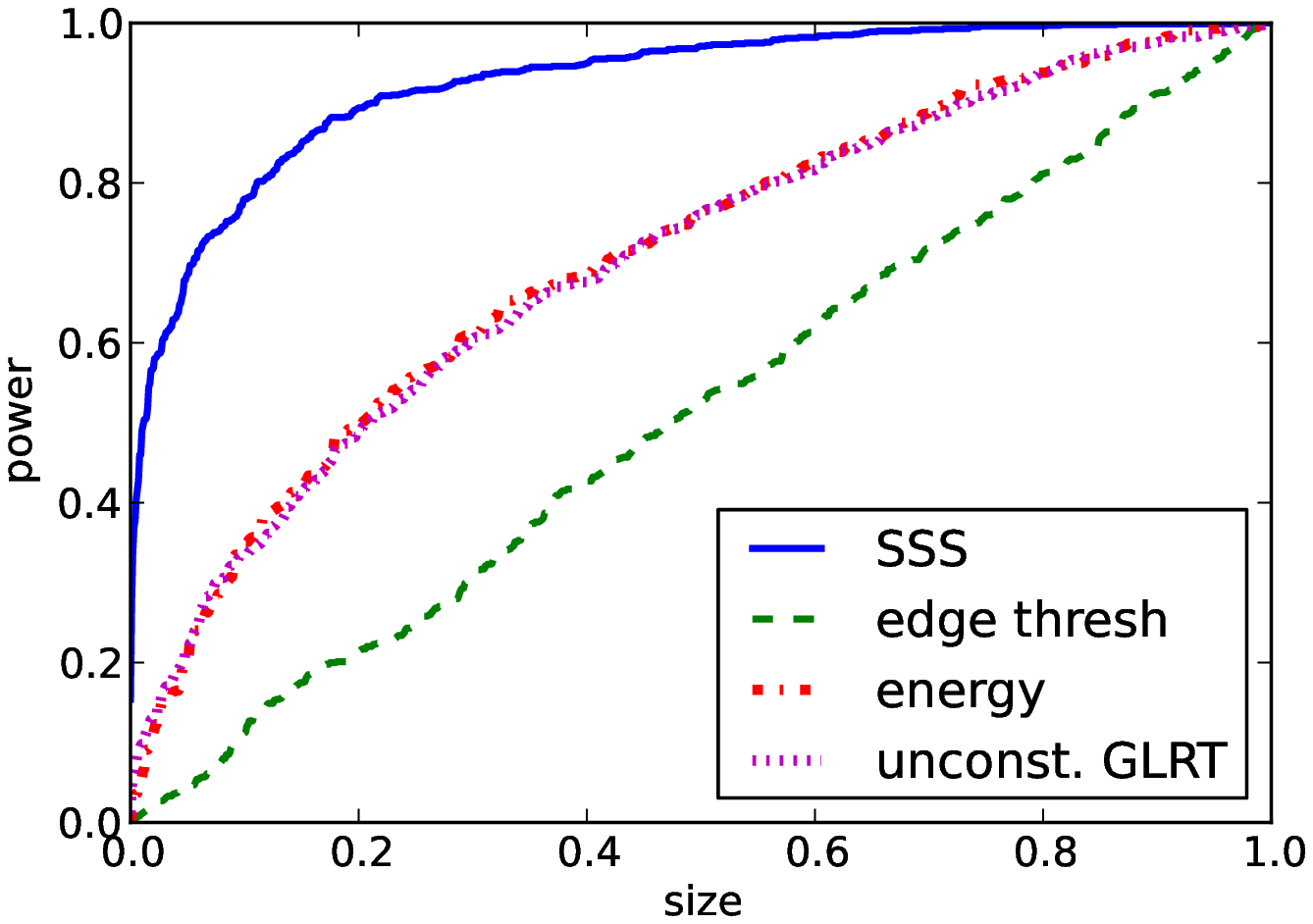}}\hspace{-.1in}
\subfigure{\includegraphics[width=2.1in]{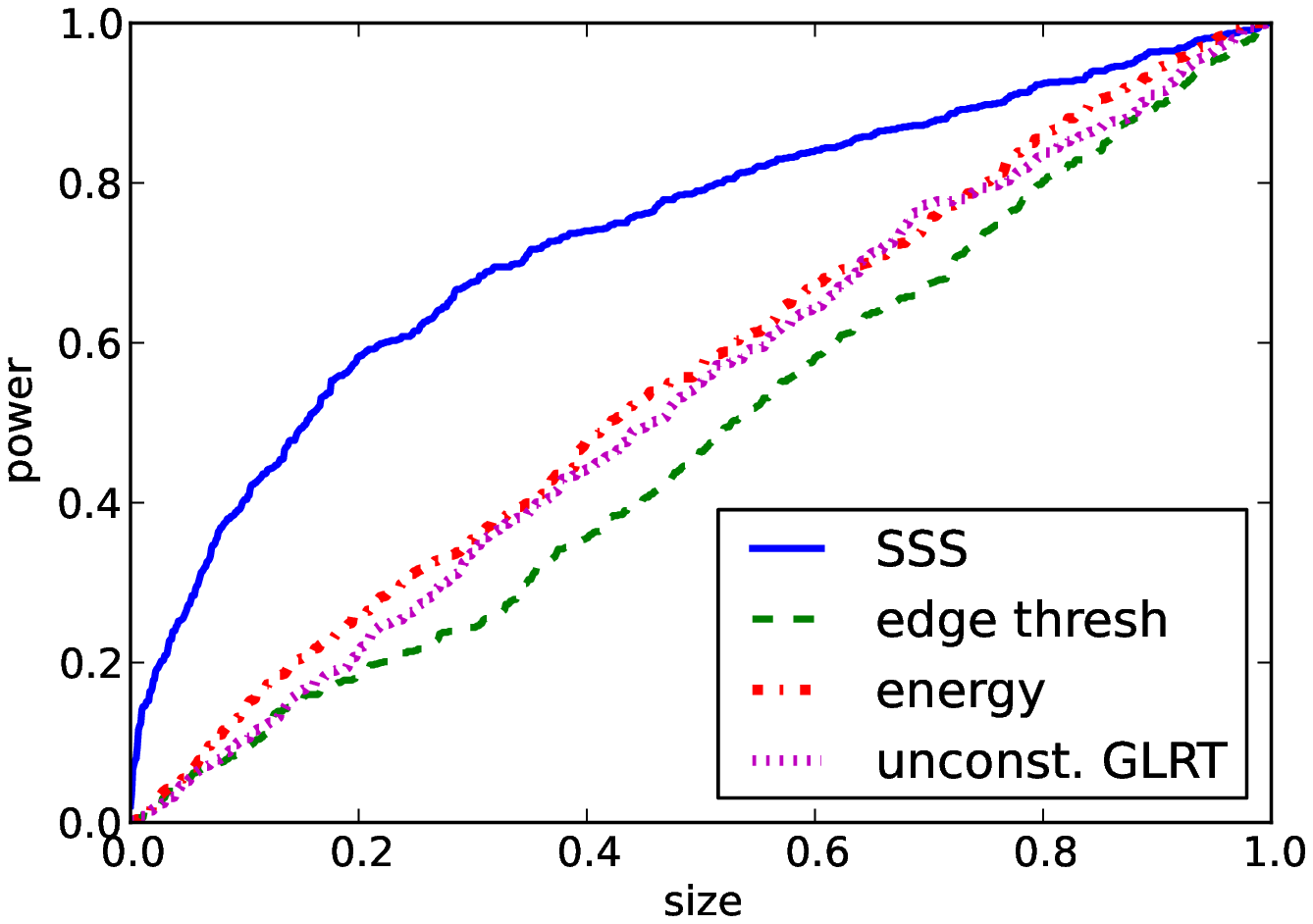}}
}
\vspace{-.1in}
\mbox{
\subfigure{\includegraphics[width=2.1in]{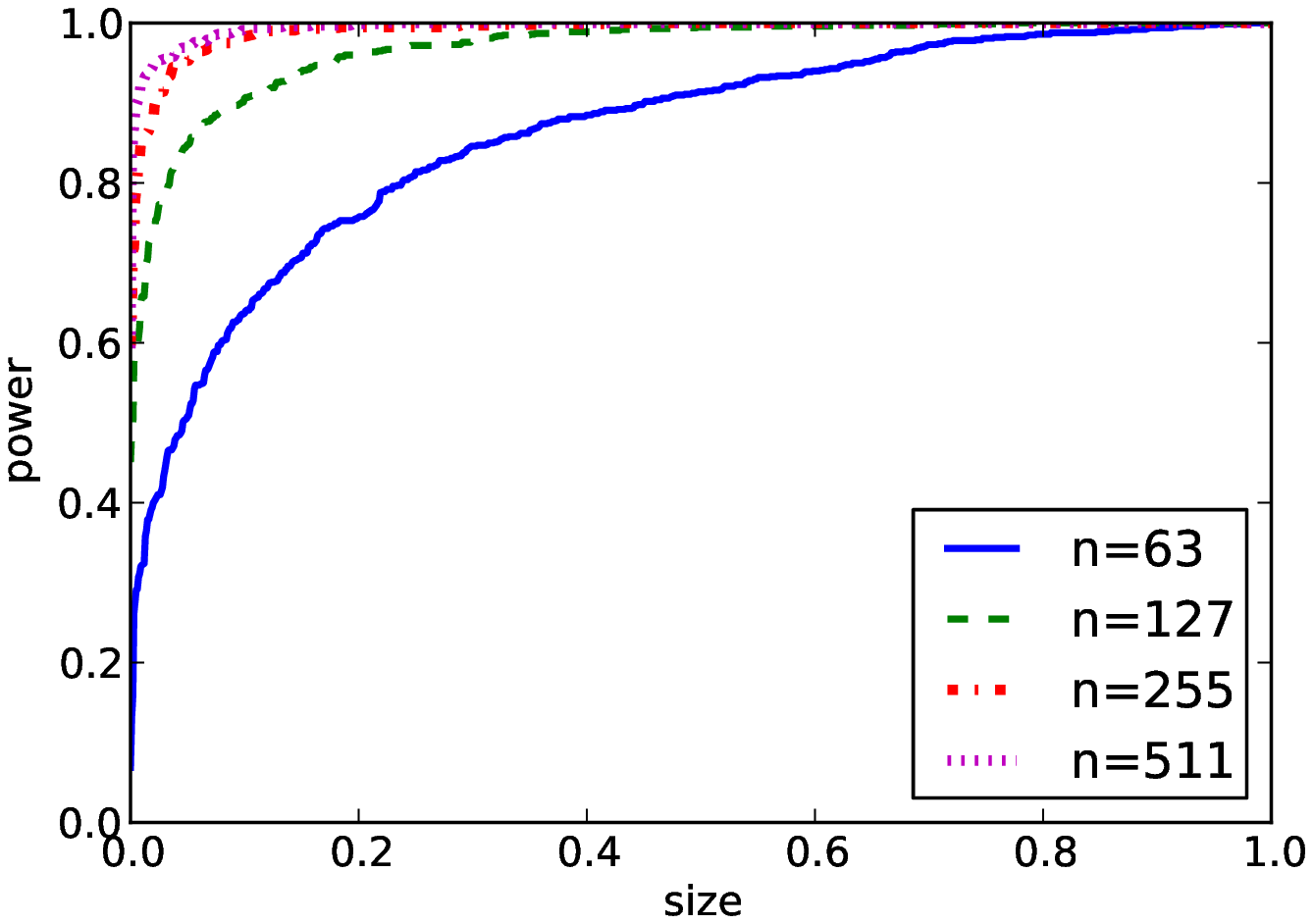}}\hspace{-.1in}
\subfigure{\includegraphics[width=2.1in]{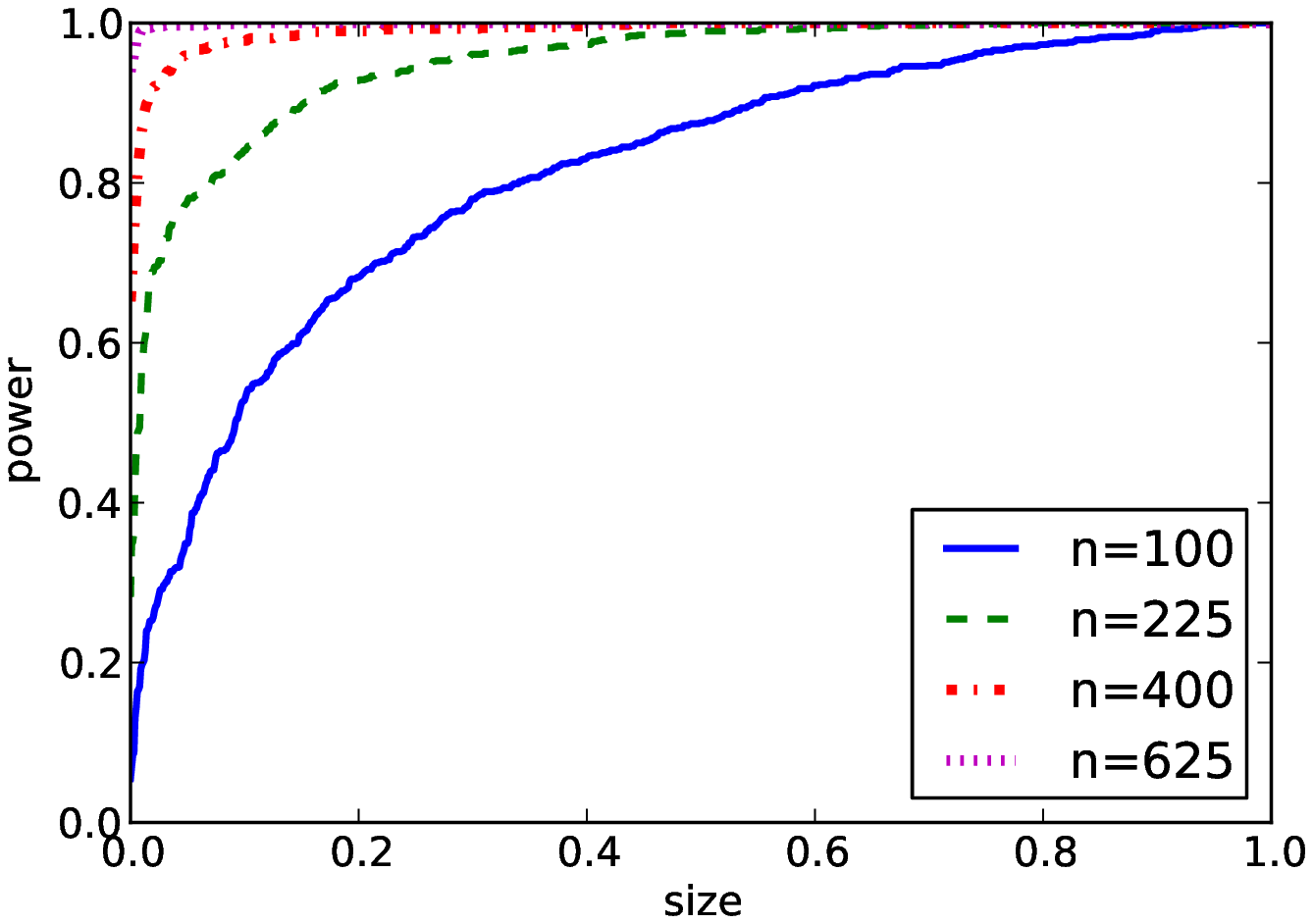}}\hspace{-.1in}
\subfigure{\includegraphics[width=2.1in]{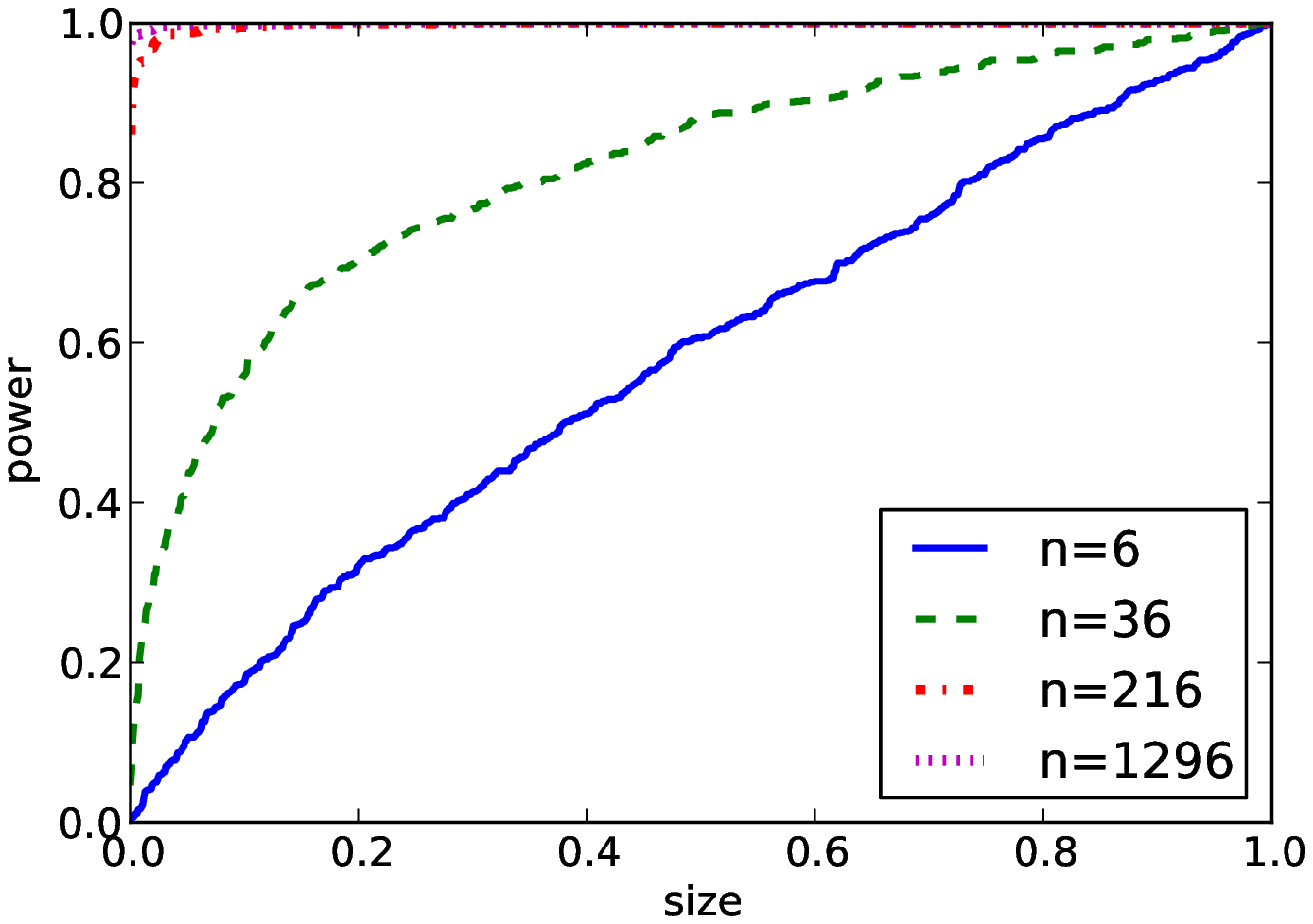}}
}
\vspace{-.1in}
\caption{\small{Above: the simulated probability of correct discovery (power) against false alarm (size) of the SSS compared to the energy detector, edge thresholding and the unconstrained GLRT of the BBT (left), Lattice (middle), and Kronecker graph (right). Below: the performance as n increases.}}
\label{fig1}
\end{figure}

\subsection{Lattice}

We will analyze the performance guarantees of the SSS over the 2-dimensional lattice graph with $p$ vertices along each dimension ($n = p^2$).
We will assume that $\rho = C n^{-1/2}$, as this is the cut sparsity of rectangles that have a low surface area to volume ratio.
By a simple Fourier analysis (see \cite{sharpnack2010identifying}), we know that the Laplacian eigenvalues are $2 (2 - \cos (2 \pi i_1/p) - \cos (2 \pi i_2 / p))$ for all $i_1,i_2 \in [p]$.
We will appeal to \eqref{eqn:dist_bound2}.
Because $1 - \cos (2 \pi i_1 /p) \approx (2 \pi i_1 /p)^2$ for $i_1 << p$, if we rewrite $i = (i_1,i_2)$ for $i_1,i_2 \in [p]$ then $\lambda_{(i_1,i_2)} \approx \frac{8 \pi^2}{n} (i_1^2 + i_2^2)$.
Hence, 
\[
k \approx |\{ (i_1, i_2) : i_1^2 + i_2^2 \le \frac{n}{8 \pi^2} \lambda_{k+1} \}| \le |\{i_1 : i_1^2 \le \frac{n}{8 \pi^2} \lambda_{k+1}\}|^2 = \lceil \frac{n}{8 \pi^2} \lambda_{k+1} \rceil
\]
Then by choosing $\lambda_{k+1} \asymp \sqrt \rho$ the term in the root of the LHS of \eqref{eqn:dist_bound2} is bounded by, $\lceil \frac{n}{8 \pi^2} \lambda_{k+1} \rceil + \frac{\rho n}{\lambda_{k+1}} \asymp n \sqrt{\rho} \asymp n^{3/4}$
modulo lower order terms.
We arrive at the following conclusion,
\begin{corollary}
\label{cor:lattice}
  For the $p \times p$ square lattice, the spectral scan statistic can asymptotically distinguish $H_0$ from signals with cut size $C n^{-1/2}$ if the SNR is stronger than,
\[
\frac{\eta}{\sigma} = \omega ( n^{3/8} )
\]
\end{corollary}

We demonstrate the improvement of the SSS over competing tests in Figure \ref{fig1}.
In these simulations a $\sqrt n / 2 \times \sqrt n / 2$ square was chosen to be $C$ with $\rho = 4/\sqrt n$.
Despite the weaker guarantee in Corollary \ref{cor:lattice} the SSS demonstrates the importance of exploiting the graph structure.

\subsection{Kronecker Graphs}

Much of the research in complex networks has focused on observing statistical phenomena that is common across many data sources.
The most notable of these are that the degree distribution obeys a power law (\cite{faloutsos1999power}) and networks are often found to have small diameter (\cite{milgram1967small}).
A class of graphs that satisfy these, while providing a simple modelling platform are the Kronecker graphs (see \cite{leskovec2007scalable,leskovec2010kronecker}).
Let $H_1$ and $H_2$ be graphs on $p$ vertices with Laplacians $\Lb_1, \Lb_2$ and edge sets $E_1, E_2$ respectively.
The Kronecker product, $H_1 \otimes H_2$, is the graph over vertices $[p] \times [p]$ such that there is an edge $((i_1,i_2),(j_1,j_2))$ if $i_1 = j_1$ and $(i_2,j_2) \in E_2$ or $i_2 = j_2$ and $(i_1,j_1) \in E_1$.
We will construct graphs that have a multi-scale topology using the Kronecker product.
Let the multiplication of a graph by a scalar indicate that we multiply each edge weight by that scalar.
First let $H$ be a connected graph with $p$ vertices.
Then the graph $G$ for $\ell > 0$ levels is defined as  
\[
\frac{1}{p^{\ell-1}} H \otimes \frac{1}{p^{\ell-2}} H \otimes ... \otimes \frac{1}{p} H \otimes H
\]
The choice of multipliers ensures that it is easier to make cuts at the more coarse scale.
Notice that all of the previous results have held for weighted graphs.

\begin{corollary}
\label{cor:kron}
  For $G$ be the Kronecker product graph described above with $n = p^\ell$ vertices, the spectral scan statistic can asymptotically distinguish $H_0$ from signals with cuts within the $k$ coarsest scale ($\rho \propto p^{2k - \ell - 1}$), if the SNR is stronger than,
\[
\frac{\eta}{\sigma} = \omega ( p^2(\ell + 2) n^{(2k + 1)/\ell}  )
\]
\end{corollary}

The proof and an explanation of $\rho$ is in the appendix.
Again, we demonstrate the improvement of the SSS over competing tests in Figure \ref{fig1}.
For these simulations the base graph $H$ was chosen to be two triangles ($K_3$) connected by a single edge (p = 6).
At the coarsest scale one of the $K_3$ subgraphs was chosen to be $C$ with $\rho = 4/ n$.


\section{Discussion}
We studied the heretofore unaddressed problem of how to tractably detect change-points in networks under Gaussian noise.
To this end we developed the spectral scan statistic, suggesting it as a computationally feasible alternative to the GLRT.
We completely characterized the performance of the SSS for any graph in terms of the spectrum of the combinatorial Laplacian.
For comparison purposes, we developed theoretical guarantees for two simple estimators.
We applied the main result to three graph models: binary balanced trees, the lattice and Kronecker graph.
We see that not only is it statistically inadmissible to ignore graph structure, but for the balanced tree the SSS gives near optimal performance.
This claim is backed by both simulation and theory.



\section*{Acknowledgements} This research is supported in part by AFOSR under grant FA9550-10-1-0382 and NSF under grant IIS-1116458.

\bibliographystyle{plainnat}
\bibliography{biblio}

\newpage
\section{Appendix}

\subsection{Proofs in Section 2}

\begin{proof}[Proof of Lemma \ref{lem:GLRTform}]
  To expedite the proof, we express the LR statistics in terms of the sufficient statistics $\yb_0 = \frac{1}{|C|} \sum_{i \in C} \yb_i \sim N(\beta_0,\sigma_0^2)$ and $\yb_1 = \frac{1}{|\bar C|} \sum_{i \in \bar C}\yb_i \sim N(\beta_1, \sigma_1^2)$ for $\sigma_0 = \sigma / \sqrt{|C|}$ and $\sigma_1 = \sigma / \sqrt{|\bar{C}|}$. Then, we obtain
\[
 2 \log \Lambda_C(\yb) = \frac{1}{\sigma_0^2} (\yb_0 - \hat{\beta})^2 + \frac{1}{\sigma_1^2} (\yb_1 - \hat{\beta})^2 
\]
where $\hat{\beta} = \frac{\sigma_1^2}{\sigma_0^2 + \sigma_1^2} \yb_0 + \frac{\sigma_0^2}{\sigma_0^2 + \sigma_1^2} \yb_1$ is the MLE under $H_0$. (The likelihood under the alternative balances with the normalizing constant of the null likelihood.) Thus,
\[
 2 \log \Lambda_C(\yb) = \frac{1}{\sigma_0^2} \left( \frac{\sigma_0^2}{\sigma_0^2 + \sigma_1^2} (\yb_0 - \yb_1) \right)^2 + \frac{1}{\sigma_1^2} \left( \frac{\sigma_1^2}{\sigma_0^2 + \sigma_1^2} (\yb_0 - \yb_1) \right)^2
\]
\[
= \frac{(\yb_0 - \yb_1)^2}{\sigma_0^2 + \sigma_1^2} = \frac{1}{\sigma^2} \frac{|C| |\bar{C}|}{|V|} (\yb_0 - \yb_1)^2
\]
\[
= \frac{1}{\sigma^2} \frac{|V|}{|C| |\bar{C}|} \left( \frac{|\bar{C}|}{|V|}\sum_{v \in C} \yb_v - \frac{|C|}{|V|} \sum_{v \in \bar{C}} \yb_v \right)^2
\]
\begin{equation}
\label{eqn:LRT_form}
= \frac{1}{\sigma^2} \frac{|V|}{|C| |\bar{C}|} \left( \sum_{v \in C} \yb_v - \frac{|C|}{|V|} \sum_{v \in V} \yb_v \right)^2 = \frac{1}{\sigma^2} \frac{|V|}{|C| |\bar{C}|} \left( \sum_{v \in C} \tilde{\yb}_v \right)^2.
\end{equation}
Now we let $\xb = {\bf 1}_C$, making the statistic above
\[
2 \sigma^2 \log \Lambda_C(\yb) = \frac{\xb^\top \tilde \yb \tilde \yb \xb}{\xb^\top \Kb \xb} \textrm{ and } \frac{|\partial C | |V| }{|C||\bar{C}|} = \frac{\xb^\top \Lb \xb}{\xb^\top \Kb \xb}.
\]
The result now follows by considering all the indicator functions corresponding to the sets in $\mathcal{C}$. 
\end{proof}

\subsection{Proofs in Section 3}

\begin{proof}[Proof of Theorem \ref{thm:lower_bd}]
Let the true $C \in \Ccal$ be known.
The performance of the optimal test with $C$ known, which by the Neyman-Pearson Lemma is based on $2 \log \Lambda_C(\yb)$, bounds the performance of that with $C$ unknown. To this end, note that, under $H_0$, the LR statistic \eqref{eq:LR} has a $\chi^2_1$, while under the alternative $H_1^C$ it has a $\chi^2_1(\lambda)$ distribution with non-centrality parameter
\[
\lambda = \frac{\delta^2}{\sigma^2}\frac{|C| |\bar C|}{|V|} = \frac{\eta^2}{\sigma^2},
\] 
which is the square of the SNR. For fixed $C$, asymptotically indistinguishable of $H_0$ versus $H_C^1$ follows by considering any threshold and noticing that the associated type 1 and type 2 errors are non-vanishing under the SNR scaling assumed in the statement. Since the risk of testing $H_0$ versus $H_1$ is no smaller than the risk of testing $H_0$ versus $H_C^1$, the result follows.
\end{proof}
We remark that the proof of the previous result shows that when distinguishing $H_0$ from $H_1^C$, the power of the test is maximal when $|C| = |\bar C|$ for a fixed value of the SNR.

\begin{proof}[Proof of Lemma \ref{lem:cheegerBdd}]
Without loss of generality, let $\yb \sim \Ncal(\zero, \Ib)$. We recall that, since $G$ is connected, the combinatorial Laplacian $\Lb$ is symmetric, its smallest eigenvalue is zero and the remaining eigenvalues are positive. By the spectral theorem, we can write $\Lb = \Ub \Lambda \Ub^\top$, where $\Lambda$ is a $(n-1) \times (n-1)$ diagonal matrix containing the positive eigenvalues of $\Lb$ in increasing order and the columns of the $n \times (n-1)$ matrix $\Ub$ are the associated eigenvectors.
Then, since each vector $\xb \in \mathbb{R}^n$ with ${\bf1}^\top \xb = 0$ can be written as $\Ub \zb$ for a unique vector $\zb \in \mathbb{R}^{n-1}$, we have
\[
\begin{array}{rcl}
\mathcal{X} & = & \{ \xb \in \mathbb{R}^n \colon \xb^\top \Lb \xb \leq \rho, \xb^\top \xb =1, {\bf1}^\top \xb \leq 0\}\\
& = & \{ \Ub \zb \in \mathbb{R}^n \colon \zb \in \mathbb{R}^{n-1}, \zb^\top \Ub^\top \Lb \Ub \zb \leq \rho, \zb^\top \Ub^\top \Ub\zb \leq1\} \\
& = & \{ \Ub \zb \in \mathbb{R}^n \colon \zb \in \mathbb{R}^{n-1}, \frac{1}{\rho} \zb^\top \Lambda \zb \leq 1, \zb^\top \zb \leq1\}, \\
\end{array}
\]
where in the third identity we have used the fact that $\Ub^\top \Ub = \Ib_{n-1}$. Letting $\mathcal{Z} = \{ \zb \in \mathbb{R}^{n-1} \colon \frac{1}{\rho} \zb^\top \Lambda \zb \leq 1, \zb^\top \zb \leq 1 \}$, we see that
\[
\sup_{\xb \in \mathcal{X}} \xb^\top \yb = \sup_{\zb \in \mathcal{Z}} \zb^\top \Ub^\top \yb  \stackrel{d}{=} \sup_{\zb \in \mathcal{Z}} \zb^\top \xib,
\]
where $ \xib \sim N(0,\Ib_{n-1})$ and $\stackrel{d}{=}$ denotes equality in distribution.

Next, we show that the set $\mathcal{Z}$, which is the intersection of an ellipsoid with the unit ball in $\mathbb{R}^{n-1}$, is contained in an enlarged ellipsoid. The supremum of the Gaussian process $\zb^\top \xib$ over $\mathcal{Z}$ will then  be bounded by the supremum of the same process over this larger but simpler set, which we will be able to bound using directly  a result from \cite{talagrand2005generic} based on chaining.
To this end, let  $\Ab = \frac{1}{\rho} \Lambda = \textrm{diag}\{ a_i\}_{i=1}^{n-1}$ and  $d = \max \{ j: a_j < 1 \}$. For for a vector $\zb \in \mathbb{R}^{n-1}$ set $\zb_1 = \zb_{[d]}$, $\zb_2 = \zb_{[n-1] \backslash [d]}$, and $\Ab_2 = \textrm{diag} \{ a_i \}_{i > d}$. Then, we observe the following chain of implications, holding for vectors $\zb \in \mathbb{R}^{n-1}$:
\[
\begin{aligned}
\| \zb \| \le 1, \zb^\top \Ab \zb \le 1 \Rightarrow \| \zb_1 \| \le 1, \sum_{i > d} a_i \zb_i^2 \le 1 \\
\Rightarrow \zb_1^\top \zb_1 + \zb_2^\top \Ab_2 \zb_2 \le 2 \Rightarrow \sum_{i} \frac{\max \{1 , a_i\}}{2} \zb_i^2 \le 1.
\end{aligned}
\]

Hence, we have the bound
\[
\EE \sqrt{\hat s} \le \EE \sup_{\zb \in \RR^{n-1}} \zb^\top \xib \textrm{ s.t. } \sum_{i} 2 \max \left\{1,a_i\right\} \xb_i^2 \le 1.
\]
Recalling that $a_i = \frac{\lambda_{i+1}}{\rho}$, for $i=1,\ldots,n-1$, where $\lambda_{i+1}$ is the $(i+1)$th eigenvalue of $\Lb$, by Proposition 2.2.1 in \cite{talagrand2005generic} the right hand side of the previous expression is bounded by $\sqrt{2 \sum_{i > 1} \min\{ 1, \rho \lambda_i^{-1}\} }$.
\end{proof}

\begin{proof}[Supplement to the proof of Theorem \ref{thm:hypo_bd}]
The following property of Gaussian processes effectively reduces the study of their supremum to the study of its expectation.
It was established by \cite{borell1975brunn} and \cite{cirel1976norms} and can be found in \cite{ledoux2001concentration}.
\begin{lemma}
  \label{lem:gauss_sup_conc}
  Consider a Gaussian process $\{Z_t \}_{t \in \Ucal}$ where $\Ucal$ is compact with respect to metric 
  \[
   d(s,t) = (\EE (Z_s - Z_t)^2)^{1/2}, \quad s,t, \in \Ucal,
  \]
 and let  $\sigma^2 \ge \sup_{t \in \Ucal} \EE Z_t^2$.
  We have that with probability at least $1 - \delta$
  \[
  \left| \sup_{t \in \Ucal} Z_t - \EE \sup_{t \in \Ucal} Z_t \right| < \sqrt{ 2 \sigma^2 \log\frac{2}{\delta}}.
  \]
\end{lemma}

Notice that the natural distance is given by $d(\xb_0,\xb_1) = (\EE ((\xb_0 - \xb_1)^\top \yb)^2)^{1/2} = \sigma \| \xb_0 - \xb_1 \|$ for $\xb_0, \xb_1 \in \Xcal$.
\end{proof}

\begin{proof}[Proof of Theorem \ref{thm:energy}]
Recall that $\tilde \yb = \Kb \yb$, where $\Kb = \Ib_n - \frac{1}{n}{\bf 1} {\bf 1}^\top$ is the orthogonal projection matrix into the $(n-1)$-dimensional linear subspace of vectors orthogonal to ${\bf 1}$. 
Under $H_0$, $\tilde \yb \sim N(0,\sigma^2 \Kb)$, and, therefore, $\| \tilde \yb \|^2 \sim \chi^2_{n-1}$, since $\textrm{tr}(\Kb) = n-1$. On the other hand,  under $H_1^C$ for  a fixed $C$, $\tilde \yb \sim N(\Kb \betab,\sigma^2 \Kb)$, where $\betab$ is given in as in \eqref{eq:beta.signal}. Thus, under $H_1^C$, $\| \tilde \yb \|^2 \sim \chi^2_{n-1}(\lambda)$, where the non-centrality parameter is given by
\begin{equation}\label{eq:chi}
\lambda =  \betab^\top \Kb \betab = \frac{1}{2} \betab^\top \Kb^\top \Kb \betab =  \| \betab - \bar \betab\|^2 \geq \frac{\eta^2}{2},
\end{equation}
where the second identity is due to the fact that $\Kb$ is symmetric and idempotent and the last inequality to our assumption on the minimal separation $\eta$ between $H_0$ and any of the alternatives. Thus,  if $\eta/\sigma = \omega(\sqrt{n-1})$, 
then $\lambda = \omega(n-1)$. Hence, 
using standard chi-square tail bounds (see for example proposition 2 of \cite{MartinSSP12}) and since the bound \eqref{eq:chi} holds uniformly over all $C \in\mathcal{C}$, it follows that the null and alternate are
asymptotically distinguishable using the test statistic $\|\tilde \yb\|$ if and only if $\frac{\eta}{\sigma} = \omega(\sqrt{n-1})$.
\end{proof}

\subsection{Proof in Section 4}

\begin{proof}[Proof of Corollary \ref{cor:BBT}]

The study of the spectra of trees really began in earnest with the work of \cite{fiedler1975eigenvectors}.
Notably, it became apparent that tree have eigenvalues with high multiplicities, particularly the eigenvalue $1$.
\cite{molitierno2000tight} gave a tight bound on the algebraic connectivity of balanced binary trees (BBT).
They found that for a BBT of depth $\ell$, the reciprocal of the smallest eigenvalue ($\lambda_2^{(\ell)}$) is 
\begin{equation}
\label{eqn:tree_eig_bound}
\begin{aligned}
\frac{1}{\lambda_2^{(\ell)}} \le 2^\ell - 2\ell + 2 - \frac{2^\ell - \sqrt{2} (2\ell -1 - 2^{\ell-1})}{2^\ell - 1 - \sqrt 2 (2^{\ell - 1} - 1)} + (3 - 2 \sqrt 2 \cos (\frac{\pi}{2\ell - 1}))^{-1} \\
\le 2^\ell + 105 I\{ \ell < 4 \}
\end{aligned}
\end{equation}
\cite{rojo2002spectrum} gave a more exact characterization of the spectrum of a balanced binary tree, providing a decomposition of the Laplacian's characteristic polynomial.
Specifically, the characteristic polynomial of $\Lb$ is given by
\begin{equation}
\label{eqn:tree_char_poly}
\det (\lambda \Ib - \Lb) = p_1^{2^{\ell - 2}}(\lambda) p_2^{2^{\ell - 3}}(\lambda) ... p_{\ell - 3}^{2^2}(\lambda) p_{\ell - 2}^2(\lambda) p_{\ell - 1}(\lambda) s_\ell(\lambda)
\end{equation}
where $s_\ell(\lambda)$ is a polynomial of degree $\ell$ and $p_i(\lambda)$ are polynomials of degree $i$ with the smallest root satisfying the bound in \eqref{eqn:tree_eig_bound} with $\ell$ replaced with $i$.
In \cite{rojo2005spectra}, they extended this work to more general balanced trees.

By \eqref{eqn:tree_char_poly} we know that at most $\ell + (\ell - 1) + (\ell - 2)2 + ... + (\ell - j)2^{j - 1} \le \ell 2^j$ eigenvalues have reciprocals larger than $2^{\ell - j} + 105 I\{ j < 4 \}$.
Let $k = \max \{ \lceil \frac{\ell}{c} 2^{\ell (1 - \alpha)} \rceil, 2^3\}$, then we have ensured that at most $k$ eigenvalues are smaller than $\rho$.
For $n$ large enough
\[
\sum_{i > 1} \min\{1, \rho \lambda_i^{-1}\} \le k + \rho \sum_{j > \log k}^{\ell} \ell 2^j 2^{\ell - j}  = k + \ell (\ell - \log k) n \rho = O(n^{1-\alpha}(\log n)^2) 
\]

\end{proof}

\begin{proof}[Proof of Corollary \ref{cor:kron}]

The Kronecker product of two matrices $\Ab, \Bb \in \RR^{n \times n}$ is defined as $\Ab \otimes \Bb \in \RR^{(n \times n) \times (n \times n)}$ such that $(\Ab \otimes \Bb)_{(i_1,i_2),(j_1,j_2)} = A_{i_1,j_1} B_{i_2,j_2}$.
Some matrix algebra shows that if $H_1$ and $H_2$ are graphs on $p$ vertices with Laplacians $\Lb_1, \Lb_2$ then the Laplacian of their Kronecker product, $H_1 \otimes H_2$, is given by $\Lb = \Lb_1 \otimes \Ib_p + \Ib_p \otimes \Lb_2$ (\cite{merris1998laplacian}).
Hence, if $\vb_1, \vb_2 \in \RR^p$ are eigenvectors, viz.~$\Lb_1 \vb_1 = \lambda_1 \vb_1$ and $\Lb_2 \vb_2 = \lambda_2 \vb_2$, then $\Lb (\vb_1 \otimes \vb_2) = (\lambda_1 + \lambda_2) \vb_1 \otimes \vb_2$, where $\vb_1 \otimes \vb_2$ is the usual tensor product.
This completely characterizes the spectrum of Kronecker products of graphs.

We should argue the choice of $\rho \propto p^{2k - \ell - 1}$, by showing that it is the results of cuts at level $k$.
We say that an edge $e = ((i_1,...,i_\ell),(j_1,...,j_\ell))$ has scale $k$ if $i_k \ne j_k$.
Furthermore, a cut has scale $k$ if each of its constituent edges has scale at least $k$.
Each edge at scale $k$ has weight $p^{k - \ell}$ and there are $p^{\ell-1}$ such edges, so cuts at scale $k$ have total edge weight bounded by 
\[
p^{\ell - 1} \sum_{i = 1}^k p^{i - \ell} = p^{k - 1} \frac{p - \frac{1}{p^{k-1}}}{p - 1} \le \frac{p^k}{p - 1}
\]
Cuts at scale $k$ leave components of size $p^{\ell - k}$ intact, meaning that $\rho \propto p^{2k - \ell - 1}$ for large enough $p$. 

We now control the spectrum of the Kronecker graph.
Let the eigenvalues of the base graph $H$ be $\{\nu_j \}_{j=1}^p$ in increasing order.
The eigenvalues of $G$ are precisely the sums
\[
\lambda_i = \frac{1}{p^{\ell-1}} \nu_{i_1} + \frac{1}{p^{\ell-2}} \nu_{i_2} + ... + \frac{1}{p} \nu_{i_{\ell-1}} + \nu_{i_\ell}
\]
for $i = (i_j)_{j = 1}^\ell \subseteq [p]$.
The eigenvalue distribution $\{ \lambda_i \}$ stochastically bounds 
\[
\lambda_i \ge \sum_{j = 1}^\ell \frac{1}{p^{\ell-j}} \nu_2 I\{\nu_{i_j} \ne 0\} \ge \frac{\nu_2}{p^{Z(i)}}
\]
where $Z(i) = \min \{j : \nu_{i_{\ell - j}} \ne 0\}$.
Notice that if $i$ is chosen uniformly at random then $Z(i)$ has a geometric distribution with probability of success $(p - 1)/p$.
Also $\rho / (\frac{\nu_2}{p^{Z(i)}}) = p^{Z(i) + 2k - \ell - 1}/\nu_2 \ge 1$ if $Z(i) \ge \ell + 1 - 2k + \log_p \nu_2$, so 
\[
\frac{1}{p^\ell}\sum_{i \in [p]^\ell} \min\{1 , \frac{\rho}{\lambda_i}\} \le \frac{p^{2k-\ell-1}}{\nu_2} + \sum_{Z = 1}^{\lfloor \ell + 1 - 2k + \log_p \nu_2 \rfloor} \frac{p^{Z + 2k - \ell - 1}}{\nu_2} \frac 1{p^Z} \frac{p - 1}{p} \le \frac{(\ell + 2) p^{2k - \ell - 1}}{\nu_2}
\] 
This followed from the geometric probability mass function.
We also know that the algebraic connectivity, $\nu_2$, is bounded from below by $4 p^{-2}$, so the following result holds.

\end{proof}

\end{document}